\documentclass{amsart}

\usepackage{amsmath,amssymb,amsthm}
\usepackage{graphicx}

\usepackage[ruled,vlined,noend]{algorithm2e}
\SetKwProg{Function}{Function}{}{}
\SetKwInOut{Input}{Input}
\SetKwInOut{Output}{Output}
\usepackage{wasysym}
\usepackage{hyperref}
\usepackage{tikz-cd}

\SetKw{True}{true}
\SetKw{False}{false}
\newcommand{\OO}{\mathcal{O}}
\newcommand{\Q}{\mathbb{Q}}
\newcommand{\N}{\mathbb{N}}
\newcommand{\Z}{\mathbb{Z}}
\newcommand{\R}{\mathbb{R}}

\newcommand{\ta}{\tilde{\alpha}}

\DeclareMathOperator{\Tr}{Tr}
\DeclareMathOperator{\cv}{cv}

\newtheorem{remark}{Remark}[section]
\newtheorem{definition}{Definition}
\newtheorem{theorem}{Theorem}
\newtheorem{proposition}{Proposition}
\newtheorem{lemma}{Lemma}

\begin{document}

\title{Explicit Euclidean division algorithms for some degree 8 number rings}
\author{Christophe Levrat}
\date{\today}

\maketitle

\begin{abstract}
This article focuses on some rings of integers of number fields which are known to be norm-Euclidean domains, but for which no explicit algorithm computing the Euclidean division has yet been studied or implemented. The rings of integers we are interested in were proven to be Euclidean by H.W. Lenstra, Jr in 1978; they include the $n$-th cyclotomic rings for $n=15,20,24$. We present an algorithm performing Euclidean division in these rings based on Lenstra's proof and a closest vector computation by Conway and Sloane, and study its complexity. We give a complete implementation of the algorithm in \textsc{SageMath}. We also estimate the size of the remainders obtained when computing Euclidean divisions with this algorithm.
\end{abstract}

\section{Introduction}

This article deals with efficient division algorithms in Norm-Euclidean rings of integers of number fields.

\paragraph{Euclidean and Norm-Euclidean rings of integers} 
An integral domain $R$ is said to be Euclidean if there exists a map $\theta\colon R-\{0\} \to \N$ 
such that for all $(a,b) \in R-\{ 0\}$, there exist $q,r\in R$ such that \[ a = bq + r \qquad \text{and}\qquad \text{either }r=0\text{ or }\theta(r)<\theta(b).\]
Any Euclidean domain is a principal ideal domain, and being able to compute such Euclidean divisions immediately allows to compute gcd's in the ring. 
\begin{remark} We will call such a map $\theta$ a \textit{Euclidean function}; it used to be called \textit{Euclidean algorithm} (see \cite{motzkin1949}, \cite{lemmermeyer2004}), a designation we will not use in order to avoid any confusion with actual algorithm computing gcd's.
\end{remark}
Denote by $F$ a number field, i.e. finite extensions of $\Q$,  and by $\OO_F$ its ring of integers. The question of Euclideanity for $\OO_F$ is twofold: \begin{itemize}
\item Is $\OO_F$ a Euclidean domain? If it is, then...
\item Is the norm $N_{F/\Q}\colon \OO_F-\{ 0\}\to \N$ a Euclidean function for $\OO_F$? If this is the case, then $\OO_F$ (and by extension $F$) is said to be \textit{Norm-Euclidean}. 
\end{itemize}
Note that a Euclidean domain may have different Euclidean functions; a generic construction of such a function for any Euclidean domain is given in \cite{motzkin1949}.

\paragraph{Cyclotomic rings} Consider a positive integer $n$ such that $n\not\equiv 2\mod 4$. Let $\zeta_n$ denote a primitive $n$-th root of unity. The
cyclotomic number field $K_n = \Q(\zeta_n)$ has degree $\varphi(n)$ over $\Q$, where $\varphi$ denotes the Euler totient function. Its ring of
integers is $\Z[\zeta_n]$.

\paragraph{Euclidean cyclotomic rings} Harper showed in \cite[Theorem C]{harper2004} that the ring $\Z[\zeta_n]$ is Euclidean if and only if it is a principal ideal domain. The complete list of the 30 integers $n\neq 2 \mod 4$ for which $\Z[\zeta_n]$ is a principal ideal domain may be found in \cite[§8]{lemmermeyer2004}, the biggest value of $n$ in this list is 84. Among these, much fewer are known to be Norm-Euclidean: for \[ n=1,3,4,5,7,8,9,11,12,13,15,16,20,24\]
the ring $\Z[\zeta_n]$ is known to be Norm-Euclidean. 
This list includes four cases where $\varphi(n) = 8$, namely $n \in
\{15,16, 20, 24\}$. The ring $\Z[\zeta_{16}]$ was proven to be norm-Euclidean by Ojala \cite{ojala77}. In the remaining cases, Lenstra's proof of Euclideanity \cite{lenstra1978} applies.
Note that $\Z[\zeta_{15}]$ and $Z[\zeta_{20}]$ had already been proven to be Euclidean in~\cite{lenstra1975}, but non constructively.
For bigger values of $n$, very little is known about norm-Euclideanity of $\Z[\zeta_n]$.

\paragraph{Effectively Euclidean number rings} Not all proven Euclidean domains come with a straightforward algorithm performing Euclidean division. In some small number rings, Euclidean division is performed by flooring (as in $\Z$: the quotient in the division of $a$ by $b$ is $\lfloor a/b\rfloor$) or by rounding (as in $\Z[i]$).
In \cite{kaiblinger2011}, Kaiblinger gives an exhaustive list of the integers $n\not\equiv 2\mod 4$ such that flooring or rounding immediately yields the quotient of the Euclidean division in $\Z[\zeta_n]$. The biggest such $n$ is 12 \cite[Theorem 1]{kaiblinger2011}. In the cases we are interested in, the approach is similar but simple flooring or rounding of the coordinates is not enough.

\paragraph{Lenstra's proof} In \cite{lenstra1978}, Hendrik Lenstra proves that the rings of integers of the fields \[ \Q(\zeta_{15}),\Q(\zeta_{20}),\Q(\zeta_{24})\text{ and }\Q(\sqrt{3},\sqrt{5},\sqrt{-1})\]
are norm-Euclidean. In his proof, he defines a bilinear product on these rings which endows them with the structure of a lattice generated by a specific root system. Using properties of this root system, he explains
how to compute the Euclidean division of $a\in R$ by $b\in R-\{ 0\}$: if $q$ denotes the quotient and $r$ the remainder, then $\frac{r}{b}$ has the same norm (w.r.t. this bilinear product) as an element of a fundamental simplex of the lattice, which is at most 1. The AM/GM inequality then ensures that $N_{F/\Q}(r/b)\leqslant 1$.

\paragraph{Our contribution} In this paper, we explain how Lenstra's proof of Euclideanity  can be reformulated in terms of the closest vector problem (CVP) in a lattice of type $E_8$. We then use Conway and Sloane's algorithm \cite{conway_sloane822} to solve this problem efficiently, and compute the worst-case complexity of the resulting algorithm. 
We provide an implementation of this algorithm in \textsc{SageMath}
\footnote{available at \url{https://anonymous.4open.science/r/euclidean_div_degree8-145E/}} \cite{sagemath}. 
Finally, we estimate how far the resulting algorithm is from being optimal, i.e. how large the remainders in the Euclidean divisions performed by this algorithm can be.

\section{Root systems and lattices}\label{sec:roots}

Root systems are a ubiquitous object in mathematics; they appear notably in the classification of Lie groups and algebraic groups.
A detailed introduction to root systems may be found in \cite[Chapter 8]{hall13}.
Let $V$ be an $\R$-vector space of finite dimension $n$, equipped with a non-degenerate positive definite bilinear form $\langle\cdot \mid \cdot \rangle$. 
Given $\alpha\in V$ and $k\in \Z$, 
we denote by $H_{\alpha,k}$ the affine hyperplane 
$\langle\alpha\mid\cdot \rangle^{-1}(k)$ orthogonal to $\alpha$, and by $s_{\alpha,k}$ the orthogonal reflection about $H_{\alpha,k}$.
We set $H_{\alpha}=H_{\alpha,0}$ and $s_\alpha=s_{\alpha,0}$.

\begin{definition}
A \textit{root system} in $V$ is a finite subset $R$ of $V$ such that: \begin{itemize}
\item $R$ spans $V$ as an $\R$-vector space;
\item for all $\alpha\in R$, $\R \alpha\cap R=\{ \pm\alpha\}$;
\item for all $\alpha,\beta\in R$, $2\frac{\langle\alpha\mid\beta\rangle}{\langle\beta\mid\beta\rangle}\in\Z$;
\item for all $\alpha,\beta\in R$, $s_\alpha(\beta)\in R$.
\end{itemize}
\end{definition}

Elements of $R$ are called \textit{roots}. The connected components of $V-\cup_{\alpha\in R}H_\alpha$ are called \textit{chambers}, and the connected components of $V-\cup_{\alpha,k}H_{\alpha,k}$
are called \textit{alcoves}. The \textit{dual root} of a root $\alpha\in R$ is $\alpha^\vee=\frac{2}{\langle\alpha\mid \alpha\rangle}\alpha$. The set $R^\vee=\{ \alpha^\vee,\alpha\in R\}$ is also a root system.

A \textit{base} of $R$ is a tuple $(\alpha_1,\dots,\alpha_n)\in R^n$ such that every element of $R$ can be written as a linear combination $\sum_{i=1}^n u_i\alpha_i$ where the $u_i$ are integers which are all of the same sign.

Once a base $B=(\alpha_1,\dots,\alpha_n)$ has been fixed, we denote by $R^+$ the set of elements of $R$ whose coordinates in this base are nonnegative. 
There exists a \textit{maximal root} $\ta=\sum_i u_i \alpha_i\in R^+$ 
such that for all $\alpha=\sum_i v_i\alpha_i\in R^+$, for all $i\in \{1\dots n\}, u_i\geqslant v_i$.

The (closed) \textit{principal alcove} with respect to $B$ is \[ A_0=\{ x\in V\mid \forall \alpha\in R^+, 0\leqslant \langle x\mid  \alpha\rangle\leqslant 1\}\]

The \textit{Weyl group} $W(R)$ of $R$ is the subgroup of elements $\sigma\in SO(V)$ such that $\sigma(R)=R$. 
It is a subgroup of the affine Weyl group $W_a(R)$, whose elements are compositions of elements of $W(R)$ with translations by elements of $R^\vee$.

The Weyl group $W(R)$ acts simply transitively on the set of chambers, and the affine Weyl group $W_a(R)$ acts simply transitively on the set of alcoves.
The abelian group generated by $R$ is a lattice $L$.
The \textit{Voronoi region} around the origin with respect to $R$ is
\[ V_0(L)=\{  x\in V\mid \forall y\in L, \Vert x\Vert \leqslant \Vert x-y\Vert\}.\]
It is the set of vectors that are closer to 0 than to any other element of $L$.
We will be interested in a particular type of root system: the type $E_8$.

\begin{definition}\label{df:E8} Let $R$ be a root system in an $8$-dimensional $\R$-vector space $V$. The root system $R$ is said to be of type $E_8$ if there is a base $(\alpha_1,\dots,\alpha_8)$ of $R$ such that the matrix $(\langle \alpha_i\mid \alpha_j\rangle )_{i,j}$ is the following:
\[ M=\begin{pmatrix}
2 & 0 & -1 & 0 & 0 & 0 & 0 & 0 \\
0 & 2 & 0 & -1 & 0 & 0 & 0 & 0 \\
-1 & 0 & 2 & -1 & 0 & 0 & 0 & 0 \\
0 & -1 & -1 & 2 & -1 & 0 & 0 & 0 \\
0 & 0 & 0 & -1 & 2 & -1 & 0 & 0 \\
0 & 0 & 0 & 0 & -1 & 2 & -1 & 0 \\
0 & 0 & 0 & 0 & 0 & -1 & 2 & -1 \\
0 & 0 & 0 & 0 & 0 & 0 & -1 & 2
\end{pmatrix}
\]
\end{definition}
In particular, every element $x$ of a root system of type $E_8$ satisfies $\langle x\mid x\rangle=2$. The Voronoi region around the origin with respect to such a root system is a polytope with 19440 vertices and 240 facets \cite[III.E]{conway_sloane82}.

\section{Euclidean division in $F$: Lenstra's proof}\label{sec:eucl}
Consider a number field $F$ of degree 8 over $\Q$ which is a totally imaginary unramified quadratic extension of a totally real number field $K$. Supppose that the different ${\rm Diff}(K/\Q)$ is generated by a totally positive element $\delta$, and that the discriminant $\Delta_K$ of $K$ is at most 4096. There are exactly four fields satisfying these conditions:
\[ \Q(\zeta_{15}), \Q(\zeta_{20}), \Q(\zeta_{24}),\text{ and } \Q(\sqrt{3},\sqrt{5},\sqrt{-1}).\] Note that this does not apply to the totally imaginary extension $\Q(\zeta_{16})$ of the totally real field $\Q(\zeta_{16}+\zeta_{16}^{-1})$, because it is ramified at a finite place.
The main result of  \cite{lenstra1978} is the following.
\begin{theorem}
Under these hypotheses, the ring of integers $\OO_F$ of $F$ is norm-Euclidean.
\end{theorem}

The outline of the proof is the following. Define an inner product on $F$ by:
\[ \langle x\mid y\rangle =\Tr_{F/\Q}\left( \frac{x\bar y}{\delta}\right).\]
Then the ring of integers $\OO_F$ of $F$ is a rank 8 lattice $\Gamma_8$, generated by the root system
$R=\{ x\in \OO_F\mid \langle x\mid x\rangle=2\}$ which is of type $E_8$.

Now consider $a,b\in \OO_F$, and set $x=ab^{-1}$. One wishes to compute $q\in \OO_F$ such that $N_{F/\Q}(x-q)<1$, because then $r=a-bq$ satisfies $N_{F/\Q}(r)<N_{F/\Q}(b)$. 

The key inequality in \cite{lenstra1978}, based on that between arithmetic and geometric mean, is the following, valid for any $z\in F$:

\[ N_{F/\Q}(z) \leqslant \frac{\Vert z\Vert ^8 \Delta_K}{4096}.\]
Therefore, the condition on the discriminant ensures that $N_{F/\Q}(z)<1 $ for all $z\in F$ such that $\Vert z\Vert \leqslant 1$.
Hence, it is enough to find $q$ such that $\Vert x-q\Vert \leqslant 1$. Lenstra's argument is the following: since the affine Weyl group $W_a(\Gamma_8)$ -- whose elements are compositions of linear isometries and translations by elements of $\Gamma_8$ -- acts transitively on the set of all alcoves, one may find $q\in\Gamma_8$ and $w\in W(\Gamma_8)$ such that $w(x-q)$ belongs to the fundamental alcove, which ensures that $\Vert x-q\Vert \leqslant 1$.

\section{Link with a closest vector problem}

Lenstra's computation of the quotient may be reformulated in terms of a closest vector problem. Indeed, the following result due to Conway and Sloane entails that the found quotient $q$ is actually a closest vector in the lattice $\Gamma_8$ to the vector $x$.

\begin{theorem}\cite[Theorem 5]{conway_sloane82}
For any root lattice $\Lambda$, the Voronoi region
around the origin is the union of the images of the principal alcove under the Weyl group $W(\Lambda)$. 
\end{theorem}

Hence, in order to perform a Euclidean division in $F$, it is enough to solve a closest vector problem for the lattice $\Gamma_8$.

\paragraph{Generic algorithms} The closest vector problem (CVP) is subject of intensive research, in particular because of its use in post-quantum cryptography. Generic methods compute a lattice point which is very close to $x$; such algorithms are implemented in computer algebra systems such as \textsc{SageMath} and \textsc{Magma}. A survey on the problems in lattice-based cryptography and these algorithms can be found in \cite{deboer_vanwoerden25}. However, the lattice point $u$ which is returned by the algorithm is not provably equal to the closest lattice point $v$. Indeed, they only guarantee that there is a constant $c>1$ such that 
\[ \Vert x-u\Vert\leqslant c\Vert x-v\Vert.\] 
Since the covering radius of $\Gamma_8$ is 1, it is possible that these approximate algorithms sometimes return a vector $u$ such that \[ \Vert x-u\Vert >1.\]

There are also deterministic procedures to compute \textit{the} closest vector to $x$ in a lattice. One of these is implemented in \textsc{SageMath}, but its performance is already prohibitive in dimension 8.

\paragraph{The case of $E_8$: an algorithm due to Conway and Sloane} There is a standard root system of type $E_8$ in $\R^8$, described for instance in \cite[Ch. 6, Planche VII]{bourbaki07}, which admits a base in which the usual inner product on $\R^8$ has the matrix $M$ given in Section \ref{sec:roots}.
For the lattice generated by this satandard root system, Conway and Sloane describe in \cite[§6]{conway_sloane822} an algorithm computing the closest lattice point to any vector of $\R^8$ in a constant number of operations. 
It is presented in Algorithm \ref{alg:closest}. 
Given a vector $x\in \R^8$, denote by $[x]\in\Z^8$ the vector whose coordinates are the closest integers to those of $x$ (taking the floor for half-integers). 
Denote by $/x]$ the vector obtained the same way as $[x]$, but where one coordinate (the first one for which the distance to the nearest integer is the highest) is rounded the wrong way. 
Denote by $\underline{\frac{1}{2}}$ the vector $\frac{1}{2}(1,\dots,1)\in\Q^8$.

\begin{algorithm}[h]
\SetAlgoLined  
\caption{\textsc{ClosestVectorE8}}\label{alg:closest}
\KwData{Vector $x\in \R^8$
}
\KwResult{Vector $y\in\Gamma_8$ which minimises $\Vert x-y\Vert$}
\hrulefill \\
$(u_1,\dots,u_8)\leftarrow[x]$\\
\If{$u_1+\dots+u_8\equiv 0\mod 2$}
{$y_0=[x]$}
\Else{$y_0=/x]$}
$(v_1,\dots,v_8)\leftarrow \left[x-\underline{\frac12}\right]$\\
 \If{$v_1+\dots+v_8\equiv 0\mod 2$}
{$y_1=\left[x-\underline{\frac12}\right]+\underline{\frac12}$}
\Else{$y_1=\left/x-\underline{\frac12}\right]+\underline{\frac12}$}
\If{$\Vert y_0-x\Vert^2 \leqslant \Vert y_1-x\Vert^2$}{
\Return $y_0$}
\Else{\Return $y_1$}
\end{algorithm}

\section{Algorithm and Complexity}

Our simple algorithm to perform Euclidean division in $F$ is the following. We have precomputed a basis $B$ of $\OO_F$ in which the trace form defined in Section \ref{sec:eucl} has the matrix $M$ defined at the end of Section \ref{sec:roots}.

\begin{algorithm}[h]
\SetAlgoLined  
\caption{\textsc{EuclideanDivision}}\label{alg:eucldiv}
\KwData{Algebraic integers $a\in\OO_F, b\in\OO_F-\{ 0\}$ given by their coordinates in the precomputed basis $B$ 
}
\KwResult{$(q,r)\in\OO_F^2$ given by their coordinates in $B$ s.t. $a=bq+r$ and either $r=0$ or $N_{F/\Q}(r)<N_{F/Q}(b)$}
\hrulefill \\
$x\leftarrow$ vector of coordinates of $\frac{a}{b}$ in basis $B$\\
$q\leftarrow \textsc{ClosestVectorE8}(x)$\\
$r\leftarrow a-bq$\\
\Return $(q,r)$
\end{algorithm}

\begin{proposition}\label{prop:comp} Given $a,b\in\OO_F$ whose coordinates in the basis $B$ are fractions of integers of bitlength at most $n$,
Algorithm \ref{alg:eucldiv} computes $q,r\in\OO_F$ such that $a=bq+r$ and $N_{F/\Q}(r)<N_{F/\Q}(b)$ in \[ \OO(M(n))=\OO(n\log(n)) \] bit operations, where $M(n)$ denotes the complexity of multiplying two $n$-bit integers.
\end{proposition}
\begin{proof} The coordinates in $B$ of $\frac{a}{b}$ still have bitlength $O(n)$. The base changes in the algorithm all require a fixed number of additions and multiplications of $O(n)$-bit integers.
Algorithm \ref{alg:closest} consists in a fixed finite number of roundings, additions and squarings. A rounding is computed using a division with remainder, hence also requires $O(M(n))$ operations \cite[Theorem 9.8]{vzg13}. The fact that $M(n)=\OO(n\log(n))$ is proven in \cite[Theorem 1.1]{harvey_vdh}.
\end{proof}

\begin{remark}
Proposition \ref{prop:comp} entails that for $a,b\in\OO_F$ of logarithmic height at most $n$, our Euclidean division algorithm requires $O(M(n))$ bit operations.
\end{remark}

The graph in Figure \ref{fig:timings} below shows the average time taken by a Euclidean division of two $n$-bit integers using our implementation in \textsc{SageMath} of Algorithm \ref{alg:eucldiv}. The benchmark was carried out on a laptop with an Intel Core i7-13800H CPU and 32GB RAM, using \textsc{SageMath} 10.7.

\begin{figure}[h]
\includegraphics[scale=0.5]{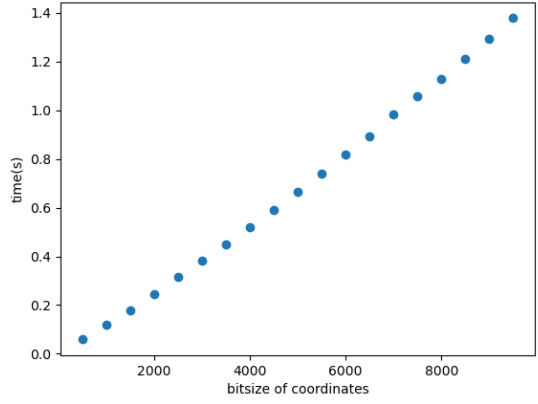}
\caption{Average time needed to compute a division of two elements in $\Z[\zeta_{15}]$ whose coordinates in the base $B$ are fractions of integers of bit length $n$} \label{fig:timings}
\end{figure}

\section{Euclidean minima and size of the computed remainders}\label{sec:eucmin}

The quotient and remainder in a Euclidean division are not necessarily unique. 
The \textit{Euclidean minimum} of a norm-Euclidean number field indicates how small the norm of the remainder can be in a Euclidean divison.

\begin{definition}
The Euclidean minimum of a number field $F$ is \[ M_F=\sup_{x\in F}\inf_{y\in x+\OO_F} |N_{F/\Q}(y)|.\]
\end{definition}

It is easy to see that if $M_F<1$, the field $F$ is norm-Euclidean, and if $M_F>1$, it is not. 
An algorithm computing Euclidean divisions in $F$ is said to be \textit{optimal} if, for all $a,b\in\OO_F$, it outputs a pair $(q,r)\in\OO_F^2$ such that $|N_{F/\Q}(r)|$ is minimal; in particular, this entails that $|N_{F/\Q}(r)|\leqslant M_F$.

The Euclidean minima of $\Q(\zeta_n)$, $n=15,20,24$ are given in the table below \cite[Table 2]{Lezowski2014}:
\begin{center} \begin{tabular}{| c| c| c| c| }
\hline
$n$ & 15 & 20 & 24 \\ \hline
$M_{\Q(\zeta_n)}$ & 1/16 & 1/5 & 1/4 \\ \hline
\end{tabular}
\end{center}

Concerning our algorithm, one may directly deduce from Lenstra's proof in \cite{lenstra1978} a first bound on the size of remainders output by our algorithm.

\begin{lemma} Let $F$ be one of the four fields to which Algorithm \ref{alg:eucldiv} applies. For all $a,b\in\OO_F$, the output $(q,r)$ of Algorithm \ref{alg:eucldiv} satisfies \[ N_{F/\Q}(r) \leqslant \frac{\Delta_K}{4096} N_{F/\Q}(b).\]
\end{lemma}
\begin{proof}
This follows directly from Lenstra's proof, by the inequality \[ \left\vert N_{F/\Q}\left(\frac{a}{b}-q\right)\right\vert\leqslant \frac{\Delta_K}{4096}\left\Vert \frac{a}{b}-q\right\Vert^8\]
and the fact that the covering radius of root systems of type $E_8$ is 1.
\end{proof}

This, however, is only an upper bound on the norm of the remainder. 
Set \begin{align*} A_F&=\sup_{x\in F} |N_{F/\Q}(x-\cv(x))|
\end{align*}
where $\cv(x)$ is the closest vector to $x$ in $\Gamma_8$ returned by Algorithm \ref{alg:closest}.
From this definition, we immediately deduce that 
\[ A_F=\sup_{a,b\in \OO_F-\{ 0\}}\frac{|N_{F/\Q}(r_{ab})|}{|N_{F/\Q}(b)|}\]
where $r_{ab}$ is the remainder in the Euclidean division of $a$ by $b$ computed using Algorithm \ref{alg:eucldiv}. We also note that
\[ A_F\leqslant \sup_{x\in \bar V_0}|N_{F/\Q}(x)|\]
where $\bar V_0$ denotes the closed Voronoi region around the origin with respect to the lattice $\Gamma_8$.
It follows from the previous discussion that \[ M_F \leqslant A_F\leqslant \frac{\Delta_K}{4096}.\]

\paragraph{Estimating $A_F$} A method to estimate $A_F$ is to numerically compute the maximum of $|N_{F/\Q}|$ on the closed Voronoi region around 0 in $F$.
Fixing a basis of the $\Q$-vector space $F$, the norm $N_{F/\Q}$ becomes a homogeneous 8-variate polynomial of degree 8. In particular, its maximum on the convex polytope $V_0$ is reached on one of its 240 facets, which are obtained by reflecting one facet around the linear hyperplanes which are orthogonal to one of the 240 roots in the roots system $\Gamma_8$.

We used \textsc{SageMath}'s constrained optimisation algorithm to compute the maximum of $N_{F/\Q}$ on the facets of $V_0$.

\begin{center}
\def\arraystretch{1.5}
\begin{tabular}{| p{2.5cm} |  p{0.6cm} | p{2cm} | p{4cm} |}
\hline
 \centering$F$
&  \centering$\!\frac{\Delta_K}{4096}$
& $\!\!\underset{x\in\bar V_0}{\sup}|N_{F/\Q}(x)|$
& $x$ attaining this $\sup$ \\ \hline

$\Q(\zeta_{15})$
& \centering $\!\frac{1125}{4096}$
&  \centering$\frac{61}{256}$
& $\frac{1}{2}(\zeta_{15}^7+\zeta_{15}^5-\zeta_{15})$ \\ \hline

$\Q(\zeta_{20})$
&  \centering$\!\frac{125}{256}$
& \centering $\frac{125}{256}$
& $\frac{1}{2}(\zeta_{20}^6+\zeta_{20}^5+\zeta_{20}^4-\zeta_{20})$ \\ \hline

$\Q(\zeta_{24})$
&  \centering$\!\frac{9}{16}$
& \centering$\frac{9}{16}$
& $\frac{1}{2}(\zeta_{24}^7+\zeta_{24}^2+\zeta_{24})$ \\ \hline

$\!\!\Q(\sqrt{3},\sqrt{5},\sqrt{-1})$
&  \centering$\!\frac{225}{256}$
& \centering $\frac{225}{256}$
& $\left(\sqrt{15}-\sqrt{3}+\sqrt{5}-1\right)\frac{\sqrt{-1}}{2}$ \newline $+\frac{1}{2}\left(\sqrt{15}+3\sqrt{3}-\sqrt{5}-7\right)$ \\ \hline
\end{tabular}
\end{center}

\begin{remark}
Interestingly, while the bound $\Delta_{K}/4096$ is attained in the cases $F=\Q(\zeta_{20}),\Q(\zeta_{24}),\Q(\sqrt{3},\sqrt{5},\sqrt{-1})$, it is not attained for $\Q(\zeta_{15})$. Hence, in the case of $\Q(\zeta_{15})$, the division algorithm yields slightly smaller remainders than one would have expected.
\end{remark}

\section{Conclusion and open questions}

Building upon ideas by H.W. Lenstra, Jr, and Conway and Sloane, we present the first efficient algorithm to compute a Euclidean division in the rings of integers of the number fields
$\Q(\zeta_{15}),\Q(\zeta_{20}),\Q(\zeta_{24})$ and $\Q(\sqrt{3},\sqrt{5},\sqrt{-1})$. It computes a Euclidean division of two algebraic integers of height $n$ in $O(n\log(n))$ bit operations ;
we provide an efficient implementation of this algorithm in \textsc{SageMath}. Depending on the field, the size of the remainder in the computed division is close to that which could be predicted from Lenstra'a article.
Here are a  few remaining open questions.
\paragraph{Optimal algorithms for these fields} In our algorithm, the size of the remainders of the Euclidean divisions is larger than the Euclidean minimum of each of the four fields. Since we based our algorithm on the only known constructive proof of norm-Euclideanity for these rings, completely new methods are required in order to develop optimal division algorithms for these fields.

\paragraph{Efficient (norm-?)Euclideanity for $\Z[\zeta_{16}],\Z[\zeta_{17}]$ and $\Z[\zeta_{19}]$} For $n\leqslant 12$, norm-Euclidean division in $\Z[\zeta_n]$ is simply given by rounding or flooring \cite{kaiblinger2011}. Norm-Euclidean division in $\Z[\zeta_{16}]$ is described by Ojala in \cite{ojala77}, in an explicit yet not trivially algorithmic way: it requires looking for some minimal element in an infinite set. The next cyclotomic rings in the list, $\Z[\zeta_{17}]$ and $\Z[\zeta_{19}]$, are proven to be Euclidean because they are principal ideal domains; however, it is only conjectured that they are norm-Euclidean. This offers several research perspectives, from trying to prove norm-Euclideanity to trying to explicitly construct a Euclidean function which could be used in a division algorithm for these rings.

\subsubsection*{Acknowledgements}
The author is deeply grateful to François Morain for introducing him to the world of Euclidean division algorithms, and for fruitful discussions along the way.

\bibliographystyle{ACM-Reference-Format}
\bibliography{cycloE8.bib}

@article{Lezowski2014,
 ISSN = {00255718, 10886842},
 author = {Lezowski, Pierre},
 journal = {Mathematics of Computation},
 number = {287},
 pages = {1397--1426},
 publisher = {American Mathematical Society},
 title = {Computation of the euclidean minimum of algebraic number fields},
 volume = {83},
 year = {2014}
}

@article{lemmermeyer2004,
  title={The Euclidean algorithm in algebraic number fields},
  author={Lemmermeyer, Franz},
  journal={Expositiones Mathematicae},
  volume={13},
  pages={385--416},
  year={1995},
  publisher={SPEKTRUM ACADEMISCHER VERLAG},
  note={Updated version available online at \url{https://www.math.purdue.edu/~jlipman/553/EuclidSurvey.pdf}}
}

@article{harper2004,
  title={$\Z[\sqrt{14}]$ is Euclidean},
  author={Harper, Malcolm},
  journal={Canadian Journal of Mathematics},
  volume={56},
  number={1},
  pages={55--70},
  year={2004},
  publisher={Cambridge University Press}
}

@article{lenstra1978,
  title={Quelques exemples d'anneaux euclidiens},
  journal={Comptes Rendus de l'Académie des Sciences de Paris},
  author={Lenstra, Hendrik W},
  volume={286},
  number={Série A},
  pages={683-685},
  year={1978}
}

@article{kaiblinger2011,
  title={Cyclotomic rings with simple Euclidean algorithm},
  author={Kaiblinger, Norbert},
  journal={JP Journal of Algebra, Number Theory and Applications},
  volume={23},
  number={1},
  pages={61--76},
  year={2011}
}

@article{motzkin1949,
  title={The euclidean algorithm},
  author={Motzkin, Theodore},
  journal={Bulletin of the American Mathematical Society},
  volume={55},
  pages={1142–1146},
  year={1949}
}

@article{conway_sloane82,
  author={Conway, J. and Sloane, N.},
  journal={IEEE Transactions on Information Theory}, 
  title={Voronoi regions of lattices, second moments of polytopes, and quantization}, 
  year={1982},
  volume={28},
  number={2},
  pages={211-226},
  doi={10.1109/TIT.1982.1056483}
}

@article{conway_sloane822,
  author={Conway, J. and Sloane, N.},
  journal={IEEE Transactions on Information Theory}, 
  title={Fast quantizing and decoding and algorithms for lattice quantizers and codes}, 
  year={1982},
  volume={28},
  number={2},
  pages={227-232},
  doi={10.1109/TIT.1982.1056484}}

@article{ojala77,
  title={Euclid's Algorithm in the Cyclotomic Field $\mathbb{Q}(\zeta_{16})$},
  author={Ojala, T},
  journal={Mathematics of Computation},
  pages={268--273},
  year={1977},
  volume={31},
  number={137}
  }

@article{lenstra1975,
  title={Euclid’s algorithm in cyclotomic fields},
  author={Lenstra Jr, Hendrik W},
  journal={J. London Math. Soc},
  volume={10},
  pages={457--465},
  year={1975}
}

@incollection{hall13,
  title={Lie groups, Lie algebras, and representations},
  author={Hall, Brian C},
  year={2013},
  booktitle={Graduate Texts in Mathematics},
  volume={290},
  publisher={Springer},
  address={Cham}
}

@book{vzg13, 
place={Cambridge}, 
edition={3}, 
title={Modern Computer Algebra}, 
publisher={Cambridge University Press},
address={Cambridge}, 
author={von zur Gathen, Joachim and Gerhard, Jürgen}, year={2013}
}

@book{bourbaki07,
  title={El{\'e}ments de Math{\'e}matique. Groupes et alg{\`e}bres de Lie: Chapitres 4, 5 et 6},
  author={Bourbaki, Nicolas},
  year={2007},
  publisher={Springer Berlin},
  address={Heidelberg}
}

@article{harvey_vdh,
  title={Integer multiplication in time O(nlog$\backslash$,n)},
  author={Harvey, David and Van Der Hoeven, Joris},
  journal={Annals of Mathematics},
  volume={193},
  number={2},
  pages={563--617},
  year={2021},
  publisher={Department of Mathematics, Princeton University Princeton, New Jersey, USA}
}

@misc{deboer_vanwoerden25,
      author = {Koen de Boer and Wessel van Woerden},
      title = {Lattice-based Cryptography: A survey on the security of the lattice-based {NIST} finalists},
      howpublished = {Cryptology {ePrint} Archive, Paper 2025/304},
      year = {2025},
      url = {https://eprint.iacr.org/2025/304}
}

@manual{sagemath,
  Key          = {SageMath},
  Author       = {{The Sage Developers}},
  Organization = {{The Sage Developers}},
  Title        = {{S}ageMath, the {S}age {M}athematics {S}oftware {S}ystem ({V}ersion 10.7)},
  note         = {{\tt https://www.sagemath.org}},
  Year         = {2025},
}

\end{document}